\documentclass[reqno]{amsart}
\usepackage{tabu}
\usepackage{amssymb}
\usepackage{mathtools}
\usepackage{a4wide,amsmath}
\usepackage{mathrsfs}
\usepackage{amsthm}
\numberwithin{equation}{section}
\numberwithin{figure}{section}
\numberwithin{table}{section}
\usepackage{bbm}
\usepackage{subfig}
\usepackage{enumerate}
\usepackage[section]{placeins}
\usepackage{graphicx}		  % Tilføjelse af billedfiler
\usepackage{ifpdf}
\ifpdf
	\DeclareGraphicsExtensions{.pdf,.eps,.jpg,.png}	
	\usepackage[suffix=]{epstopdf}
\fi
\usepackage{xcolor}
\usepackage[utf8]{inputenc}
\usepackage{hyperref}
\hypersetup{hidelinks}

\long\def\MSC#1\EndMSC{\def\arg{#1}\ifx\arg\empty\relax\else
     {\narrower\noindent%
{2020 Mathematics Subject Classification}: #1\\} \fi}
\long\def\PACS#1\EndPACS{\def\arg{#1}\ifx\arg\empty\relax\else
     {\narrower\noindent%
{PACS numbers}: #1}\fi}
\long\def\KEY#1\EndKEY{\def\arg{#1}\ifx\arg\empty\relax\else
	{\narrower\noindent% 
Keywords: #1\\}\fi}

\theoremstyle{plain}
\newtheorem{theorem}{Theorem}[section]

\theoremstyle{definition}

\theoremstyle{remark}
\newtheorem{remark}[theorem]{Remark}
\newcommand{\norm}[1]{\lVert#1\rVert}
\newcommand{\abs}[1]{\lvert#1\rvert} 
\newcommand{\inner}[1]{\langle#1\rangle}

\newcommand{\spanm}{\mathop{\textup{span}}}
\newcommand{\di}{\mathrm{d}}   % differential

\newcommand{\N}{\mathbb{N}}

\begin{document}

\title[Reconstruction from finite data]{Reconstruction in the Calder\'on problem on a fixed partition from finite and partial boundary data}

\author[H.~Garde]{Henrik Garde}
\address[H.~Garde]{Department of Mathematics, Aarhus University, Aarhus, Denmark.}
\email{garde@math.au.dk}

\begin{abstract}
	This short note modifies a reconstruction method by the author (Comm.~PDE, 45(9):1118--1133, 2020), for reconstructing piecewise constant conductivities in the Calder\'on problem (electrical impedance tomography). In the former paper, a layering assumption and the local Neumann-to-Dirichlet map were needed since the piecewise constant partition also was assumed unknown. Here I show how to modify the method in case the partition is known, for general piecewise constant conductivities and only a finite number of partial boundary measurements. Moreover, no lower/upper bounds on the unknown conductivity are needed.
\end{abstract}

\maketitle

\KEY
Calder\'on problem,
electrical impedance tomography, 
partial data reconstruction.
\EndKEY

\MSC
35R30, % PDE: inverse problems
35Q60, % PDE: optics and electromagnetic theory
35R05, % PDE: discontinuous coefficient or data
47H05. % Operator theory: monotone operators and generalizations 
\EndMSC

\section{Introduction}

In~\cite{Garde2020b} a direct reconstruction method is given for the Calder\'on problem with piecewise constant layered conductivities. Here, successively, each layer is reconstructed and the conductivity values on each component of the layer are determined by simple one-dimensional optimization problems. This approach is improved in~\cite{Garde2022}, using the theory of inclusion detection for extreme conductivity values, simplifying the involved test-operators and removing the need for bounds on the conductivity.

In the aforementioned papers, infinitely many partial boundary measurements are used in the form of the local Neumann-to-Dirichlet (ND) map. This is needed because the setting is infinite-dimensional, in the sense that the piecewise constant partition (and even the number of ``pixels'') is also assumed unknown, and must be reconstructed. Later~\cite{Harrach2023} expanded on this approach, showing that for a \emph{fixed known partition}, a piecewise constant conductivity can be found as the unique solution to a semidefinite optimization problem, using just a finite number of measurements. However, this new approach also comes with the caveat that there is an unknown set of coefficients for the optimization problem depending on the domain, the partition, and lower/upper bounds on the conductivity. Both in~\cite{Harrach2023} and the related paper~\cite{Alberti2025}, it is mentioned that the approach from~\cite{Garde2020b,Garde2022} adapts to their settings, but requires infinitely many measurements. 

In this short note, I will show how the second part of the algorithm in~\cite{Garde2020b,Garde2022} can be adapted to reconstruction of a general piecewise constant conductivity on a given fixed partition, and moreover, that up to an arbitrarily small error, only a finite number of boundary measurements are needed.

Hence, there are a number of advantages/disadvantages compared to the approach in~\cite{Harrach2023}:
\begin{itemize}
	\item It seems preferable to reconstruct the entire conductivity function simultaneously as in~\cite{Harrach2023}, although one should note that in \cite{Harrach2023} the coefficients of the optimization problem are not immediately available. This is unlike in the approach given in this manuscript, where the involved one-dimensional optimization problems are readily available and very simple.
	\item In this approach, no lower/upper bounds are needed on the unknown conductivity.
	\item This approach adapts to \emph{region-of-interest} (ROI) problems. If there is only interest in the conductivity in part of the domain, one can reconstruct a path from the domain boundary to the ROI, thus avoiding reconstruction in the remaining part to the domain. 
\end{itemize}

\section{The setting}

Let $\Omega\subset \mathbb{R}^d$, $d\geq 2$, be a bounded Lipschitz domain. Let $\nu$ be an outer unit normal to $\Omega$, and let $\Gamma\subseteq\partial\Omega$ be a non-empty relatively open subset where the boundary measurements are taken. 

Consider the partial data conductivity problem, 
\begin{equation} \label{eq:condeq}
	-\nabla\cdot(\sigma\nabla u) = 0 \quad\text{in } \Omega, \qquad \nu\cdot(\sigma\nabla u)|_{\partial\Omega} = \begin{cases}
		f &\text{on } \Gamma, \\
		0 &\text{on } \partial\Omega\setminus\Gamma,
	\end{cases}
\end{equation}
where $f$ belongs to
\begin{equation*}
	L^2_\diamond(\Gamma) = \{\, f\in L^2(\Gamma) : \inner{f,1} = 0 \,\}.
\end{equation*}
Here $\inner{\,\cdot\,,\,\cdot\,}$ refers to the usual inner product on $L^2(\Gamma)$. For the electric potential~$u$ we also enforce a $\Gamma$-mean free condition. A conductivity coefficient $\sigma$ can in general be nonnegative and measurable. We assume that $\sigma$ can formally equal zero or infinity on Lipschitz sets; such sets are called extreme inclusions (c.f.~\cite{Garde2020a}). Away from the extreme inclusions, $\sigma$ is assumed to be bounded away from zero and infinity. Note that the theory in~\cite{Garde2020a} also adapts to the case where such extreme inclusions touch the domain boundary, provided the non-extreme part is connected to $\Gamma$, i.e.~can be ``seen'' by the boundary measurements; the proofs are related to the weak form of the forward problem that readily adapts to this setting.

This gives rise to a local Neumann-to-Dirichlet (ND) map on $\Gamma$, which is a compact self-adjoint operator $\Lambda(\sigma)\colon f\mapsto u|_\Gamma$ on $L^2_\diamond(\Gamma)$. In the following, we will use ND maps for various conductivity coefficients taking the role of $\sigma$. We will denote the \emph{unknown} conductivity as $\gamma$, that we seek to reconstruct from local boundary measurements.

We assume that 
\begin{equation*}
	\overline{\Omega} = \bigcup_{j=1}^n \overline{P_j},
\end{equation*}
where $P_1,\dots,P_n\subseteq\Omega$ are non-empty pairwise disjoint Lipschitz domains. The sets are ordered such that $\partial P_1\cap \Gamma$ contains a non-empty open boundary piece, and for any $m\in\{1,\dots,n\}$ then
\begin{equation*}
	Q_m = \bigcup_{j=1}^m \overline{P_j}
\end{equation*}
has connected interior. See Figure~\ref{fig:fig1} for some examples; of course the sets $\{P_j\}_{j=1}^n$ are not required to be ``pixel-shaped'', but can be arbitrary Lipschitz domains.

The unknown conductivity $\gamma$ is assumed to be piecewise constant, satisfying
\begin{equation*}
	\gamma|_{P_j} \equiv \gamma_j
\end{equation*}
for some $\gamma_j \in (0,\infty)$, $j=1,\dots,n$. 

\begin{figure}[htb]
	\centering
	\includegraphics[width=\textwidth]{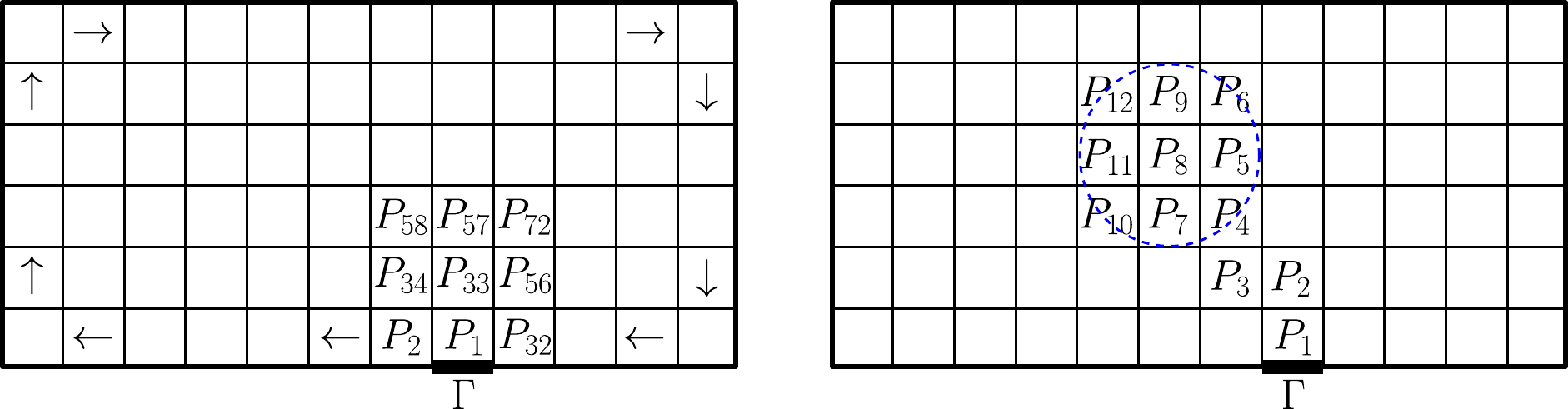}
	\caption{Examples of orderings of the $P_j$ sets. Left: Reconstructing the outermost pixels first, then proceeding to the next ``layer'' of pixels. Right: Example of a ROI (blue dashed outline), where the reconstruction is adapted to this.}
	\label{fig:fig1}
\end{figure}

We will now proceed to show how to reconstruct $\gamma_1,\dots,\gamma_n$ successively. We start by showing how this is done for infinitely many measurements $\Lambda(\gamma)$ in Section~\ref{sec:infinite}, and afterwards show that this can be reduced to finitely many measurements in Section~\ref{sec:finite}.

\section{The reconstruction method for infinite measurements} \label{sec:infinite}

The method starts at $P_1$, reconstructs $\gamma_1$ via a one-dimensional optimization problem, and moves on to $P_2$. Successively all of $\gamma$ is reconstructed on $P_1,\dots,P_n$. Hence, we will assume that we know $\gamma$ on $Q_{m-1}$ (with the convention that $Q_0 = \emptyset$) and now need to reconstruct $\gamma_m$ on $P_m$. 

We let 
\begin{equation*}
	\Gamma_m = (\partial Q_m \cap \Gamma)^\circ,
\end{equation*}
where the interior is taken relative to $\partial\Omega$. Now define $\mathcal{R}_m \colon L^2_\diamond(\Gamma) \to L^2_\diamond(\Gamma_m)$ by
\begin{equation*}
	\mathcal{R}_m f = \begin{dcases}
		f - \abs{\Gamma_m}^{-1}\inner{f,1}_{L^2(\Gamma_m)} & \text{on } \Gamma_m \\
		0 & \text{elsewhere}.
	\end{dcases}
\end{equation*}
A short computation reveals that $\mathcal{R}_m$ is the orthogonal projection onto $L^2_\diamond(\Gamma_m)$, when $L^2_\diamond(\Gamma_m)$ is considered a subspace of $L^2_\diamond(\Gamma)$ via extension by zero. We can now define
\begin{equation*}
	\Lambda_m(\gamma) = \mathcal{R}_m\Lambda(\gamma)\mathcal{R}_m,
\end{equation*}
which corresponds to the local ND map measured on $\Gamma_m$ (when restricted to this set). This is used for the compatibility with the following simulated measurement operators, where initially $\Gamma_m$ may be smaller than $\Gamma$ and is the only accessible part of the boundary for those simulated operators.

For $t>0$ and with $\mu$ either representing $0$ (perfectly insulating) or $\infty$ (perfectly conducting), we denote by $\Lambda_{m,\mu}(t)$ the local ND map (on $\Gamma_m$) for the conductivity
\begin{equation*}
	\begin{dcases}
		\mu &\text{in } \overline{\Omega}\setminus Q_m, \\
		t & \text{in } P_m, \\
		\gamma &\text{in } Q_{m-1}.
	\end{dcases}
\end{equation*}
Applying \cite[Lemma A.1]{Garde2020a} and \cite[Cases (a) and (b) in the proof of Theorem~3.7]{Garde2020a} entail the following monotonicity relations, in terms of the Loewner order of the involved operators on $L^2_\diamond(\Gamma_m)$:
\begin{align}
	t &\geq \gamma_m \qquad \text{if and only if} \qquad \Lambda_m(\gamma) \geq \Lambda_{m,\infty}(t), \label{eq:recon2a}\\
	t &\leq \gamma_m \qquad \text{if and only if} \qquad \Lambda_{m,0}(t) \geq \Lambda_m(\gamma). \label{eq:recon2b}
\end{align}
Hence, one can solve either of the following one-dimensional optimization problems:
\begin{equation*}
	\gamma_m = \min\{\, t>0 : \Lambda_m(\gamma) \geq \Lambda_{m,\infty}(t) \,\}
\end{equation*}
or
\begin{equation*}
	\gamma_m = \max\{\, t>0 : \Lambda_{m,0}(t) \geq \Lambda_m(\gamma) \,\}.
\end{equation*}
Note that \eqref{eq:recon2a} and \eqref{eq:recon2b} show that $t = \gamma_m$ is the unique value where it changes from ``the operator inequality is satisfied'' to ``the operator inequality is \emph{not} satisfied''.
\begin{remark}
	Note the interesting fact, that the approach using $\Lambda_{m,0}(t)$ corresponds to ``building'' the domain one pixel $P_m$ at a time, since having a perfectly insulating conductivity coefficient corresponds to deleting that part of the domain in the PDE problem. 
\end{remark}
\begin{remark}
	If there are known uniform lower/upper bounds $0<\alpha\leq\gamma\leq \beta$, then one can replace $0$ by $\alpha$ and $\infty$ by $\beta$ in the definitions of $\Lambda_{m,0}$ and $\Lambda_{m,\infty}$ and restrict $t\in[\alpha,\beta]$. 
	
	Moreover, this avoids the need for defining $\Gamma_m$ and using $\mathcal{R}_m$, since these are only introduced for the sake of $\Lambda_{m,0}$ and $\Lambda_{m,\infty}$, with a conductivity possibly taking an extreme value on part of the original measurement set $\Gamma$.
\end{remark}

\section{The reconstruction method for finite measurements} \label{sec:finite}

Let $\{g_i\}_{i\in\N}$ be an orthonormal basis for $L^2_\diamond(\Gamma)$ and let $\mathcal{P}_M$ be the orthogonal projection onto
\begin{equation*}
	\spanm\{g_i\}_{i=1}^M.
\end{equation*}
We will now show that, up to an arbitrarily small error $\epsilon>0$, there exists an $M\in\N$, such that the operator inequalities $\Lambda_m(\gamma) \geq \Lambda_{m,\infty}(t)$ or $\Lambda_{m,0}(t) \geq \Lambda_m(\gamma)$ can be replaced by the finite-dimensional versions given by
\begin{equation*} 
	\mathcal{P}_M\bigl(\Lambda_m(\gamma) - \Lambda_{m,\infty}(t)\bigr)\mathcal{R}_m\mathcal{P}_M \geq 0
\end{equation*}
or 
\begin{equation*} 
	\mathcal{P}_M\bigl(\Lambda_{m,0}(t) - \Lambda_m(\gamma)\bigr)\mathcal{R}_m\mathcal{P}_M \geq 0.
\end{equation*}
In particular, the following theorem gives an $M_m$ for each $m\in\{1,\dots,n\}$ needed for the different optimization problems, and we may therefore use the finite number
\begin{equation*}
	M = \max_m M_m
\end{equation*}
of boundary measurements.
\begin{theorem}
	For any $\epsilon>0$, there exists $M_m\in\N$ such that for all $M\geq M_m$, in the setting of Section~\ref{sec:infinite},  
	\begin{alignat}{4}
		t &\geq \gamma_m   &&\text{implies}\qquad   &\mathcal{P}_{M}\bigl(\Lambda_m(\gamma) - \Lambda_{m,\infty}(t)\bigr)\mathcal{R}_m\mathcal{P}_{M} &\geq 0, \label{eq:finite}\\
		t &< \gamma_m-\epsilon \qquad&&\text{implies}\qquad  &\mathcal{P}_{M}\bigl(\Lambda_m(\gamma) - \Lambda_{m,\infty}(t)\bigr)\mathcal{R}_m\mathcal{P}_{M} &\not\geq 0, \label{eq:finiteb}\\
		t &\leq \gamma_m  &&\text{implies}\qquad  &\mathcal{P}_{M}\bigl(\Lambda_{m,0}(t) - \Lambda_m(\gamma)\bigr)\mathcal{R}_m\mathcal{P}_{M} &\geq 0, \label{eq:finite2} \\
		t &> \gamma_m+\epsilon  &&\text{implies}\qquad  &\mathcal{P}_{M}\bigl(\Lambda_{m,0}(t) - \Lambda_m(\gamma)\bigr)\mathcal{R}_m\mathcal{P}_{M} &\not\geq 0. \label{eq:finite2b}
	\end{alignat}
\end{theorem}
\begin{proof}
	We will focus on \eqref{eq:finite}--\eqref{eq:finiteb}, since the proof related to \eqref{eq:finite2}--\eqref{eq:finite2b} is essentially the same. 
	
	If $t\geq \gamma_m$ then, by the self-adjointness of $\mathcal{P}_{M}$ and $\mathcal{R}_m$, \eqref{eq:recon2a} gives
	\begin{equation*}
		\inner{\mathcal{P}_{M}\bigl(\Lambda_m(\gamma) - \Lambda_{m,\infty}(t)\bigr)\mathcal{R}_m\mathcal{P}_{M}f,f} = \inner{\bigl(\Lambda_m(\gamma) - \Lambda_{m,\infty}(t)\bigr)\mathcal{R}_m\mathcal{P}_{M}f,\mathcal{R}_m\mathcal{P}_{M}f} \geq 0
	\end{equation*}
	for all $f\in L^2_\diamond(\Gamma)$ and any $M\in\N$. 
	
	Now assume $t < \gamma_m-\epsilon$. Let $\Lambda_\infty$ be the ND map (on $\Gamma_m$) for the conductivity
	\begin{equation} \label{eq:gammatmp}
		\begin{dcases}
			\infty &\text{in } \overline{\Omega}\setminus Q_m, \\
			\gamma &\text{in } Q_{m}.
		\end{dcases}
	\end{equation}
	Let $f\in L^2_\diamond(\Gamma_m)$. By \cite[Lemma~A.1 parts (ii) and (i)]{Garde2020a},
	\begin{align}
		\inner{\bigl(\Lambda_m(\gamma) - \Lambda_{m,\infty}(t)\bigr)f,f} &= \inner{\bigl(\Lambda_m(\gamma) - \Lambda_{\infty}\bigr)f,f} + \inner{\bigl(\Lambda_\infty - \Lambda_{m,\infty}(t)\bigr)f,f} \notag\\
		&\leq K\int_{\Omega\setminus Q_m}\abs{\nabla u_f}^2\,\di x + (t-\gamma_m)\int_{P_m}\abs{\nabla \widehat{u}_f}^2\,\di x \notag \\
		&\leq K\int_{\Omega\setminus Q_m}\abs{\nabla u_f}^2\,\di x -\epsilon\int_{P_m}\abs{\nabla \widehat{u}_f}^2\,\di x. \label{eq:ineqthm}
	\end{align}
	Here $K>0$ is independent of $f$ and $t$, $u_f$ is the electric potential with Neumann condition $f$ and conductivity $\gamma$, while $\widehat{u}_f$ is the electric potential with Neumann condition $f$ and the conductivity from \eqref{eq:gammatmp}. Hence, by the theory of localized potentials \cite[Lemmas~5.2 and~5.3]{Garde2020a}, there exists a sequence $(f_i)$ in $L^2_\diamond(\Gamma_m)$ such that 
	\begin{equation} \label{eq:upperlim}
		\lim_{i\to\infty} \Bigl[K\int_{\Omega\setminus Q_m}\abs{\nabla u_{f_i}}^2\,\di x -\epsilon\int_{P_m}\abs{\nabla \widehat{u}_{f_i}}^2\,\di x\Bigr] = -\infty.
	\end{equation}
	Pick an index $i_0$ such that the expression in the square brackets of \eqref{eq:upperlim} is negative and let $g = f_{i_0}$. Note that
	\begin{equation*}
		\norm{g - \mathcal{R}_m\mathcal{P}_M g} = \norm{\mathcal{R}_m(g - \mathcal{P}_M g)} \leq \norm{g - \mathcal{P}_M g} \to 0 \text{ for } M\to\infty.
	\end{equation*}
	By continuity of $f\mapsto u_f$ and $f\mapsto \widehat{u}_f$, there exists an $M_m\in\N$ such that for all $M\geq M_m$,
	\begin{equation*} 
		K\int_{\Omega\setminus Q_m}\abs{\nabla u_{\mathcal{R}_m\mathcal{P}_Mg}}^2\,\di x -\epsilon\int_{P_m}\abs{\nabla \widehat{u}_{\mathcal{R}_m\mathcal{P}_Mg}}^2\,\di x < 0.
	\end{equation*}
	Hence, for any $t < \gamma_m - \epsilon$ and any $M\geq M_m$, \eqref{eq:ineqthm} gives
	\begin{equation*}
		\inner{\mathcal{P}_M\bigl(\Lambda(\gamma) - \Lambda_{m,\infty}(t)\bigr)\mathcal{R}_m\mathcal{P}_Mg,g} = \inner{\bigl(\Lambda(\gamma) - \Lambda_{m,\infty}(t)\bigr)\mathcal{R}_m\mathcal{P}_Mg,\mathcal{R}_m\mathcal{P}_Mg} < 0. \qedhere
	\end{equation*}
\end{proof}

\subsection*{Acknowledgements}

This research is supported by grant 10.46540/3120-00003B from Independent Research Fund Denmark.

\bibliographystyle{plain}

\end{document}